\documentclass[a4paper,11pt]{article}
\usepackage[a4paper, total={15cm, 23.7cm}]{geometry}
\usepackage{amssymb,amsthm,amsmath,amsfonts,color,tikz}
\usepackage{comment}
\usepackage[normalem]{ulem}

\newtheorem{theorem}{Theorem}

\newcommand{\ds}{\displaystyle}

\newcommand{\R}{\mathbb{R}}
\newcommand{\N}{\mathbb{N}}

\def\cp{\mathrm{cap}\,}
\def\diam{\mathrm{diam}}

\newcommand{\Om}{\Omega}

\newcommand{\vps}{\varepsilon}

\newcommand{\bp}{\begin{proof}}
\newcommand{\ep}{\end{proof}}

\begin{document}
\title{On capacity and torsional rigidity}

\author{{M. van den Berg} \\
School of Mathematics, University of Bristol\\
Fry Building, Woodland Road\\
Bristol BS8 1UG\\
United Kingdom\\
\texttt{mamvdb@bristol.ac.uk}\\
\\
G. Buttazzo\\
Dipartimento di Matematica\\
Universit\`a di Pisa\\
Largo B. Pontecorvo 5, 56127 Pisa\\
Italy\\
\texttt{giuseppe.buttazzo@dm.unipi.it}}
\date{4 September 2020}
\maketitle
\vskip1truecm\indent

\begin{abstract}\noindent
We investigate extremal properties of shape functionals which are products of Newtonian capacity $\cp(\overline{\Om})$, and powers of the torsional rigidity $T(\Om)$, for an open set $\Om\subset \R^d$ with compact closure $\overline{\Om}$, and prescribed Lebesgue measure. It is shown that if $\Om$ is convex then $\cp(\overline{\Om})T^q(\Om)$ is (i) bounded from above if and only if $q\ge 1$, and (ii) bounded from below and away from $0$ if and only if $q\le \frac{d-2}{2(d-1)}$. Moreover a convex maximiser for the product exists if either $q>1$, or $d=3$ and $q=1$. A convex minimiser exists for $q< \frac{d-2}{2(d-1)}$. If $q\le 0$, then the product is minimised among all bounded sets by a ball of measure $1$.
\end{abstract}


\section{Introduction and main results \label{sec1}}

Several classical inequalities of mathematical physics are of the following form. Let $F$ and $H$ be strictly positive set functions defined on a suitable collection $\mathfrak{C}$ of open sets in $\R^d$, and which satisfy scaling relations
$$F(t\Om)=t^{\beta_1}F(\Om),\qquad H(t\Om)=t^{\beta_2}H(\Om),\quad t>0,$$
where $t\Om$ is homothety of $\Om$, and $\beta_1,\beta_2$ are constants. Then the shape functional
$$G(\Om)=H(\Om)F(\Om)^{-\beta_2/\beta_1},$$
is invariant under homotheties, and in some cases this quantity is minimal (respectively maximal) for some open set $\Om^*\in\mathfrak{C}$,
$$G(\Om)\ge G(\Om^*)\quad\text{\big(respectively $G(\Om)\le G(\Om^*)$\big)},\qquad\Om\in\mathfrak{C}.$$
 The Faber-Krahn, Krahn-Szeg{\"o}, and Kohler-Jobin inequalities are of this form. See for example the seminal text \cite{PSZ}. In a recent paper \cite{vdBBP} a more general set of inequalities was investigated. These are of the following form: let $q\in\R,$ and consider the shape functional
$$G(\Om)=H(\Om)F(\Om)^q.$$
Then, unless $q=-\beta_2/\beta_1$, this product is not scaling invariant. However, denoting by $|\Om|$ the Lebesgue measure of $\Om$, the quantity
$$\frac{H(\Om)F(\Om)^q}{|\Om|^{(\beta_2+q\beta_1)/d}}$$
is scaling invariant. The case where $H$ is the principal Dirichlet eigenvalue, and $F$ is the torsional rigidity was analysed in \cite{vdBBP}. In the present paper we investigate, in the spirit of \cite{PSZ}, the case where $H$ is the Newtonian capacity, and $F$ is the torsional rigidity. Since the Newtonian capacity is most easily defined for compact subsets of $\R^d,\,d\ge 3,$  we restrict ourselves to open sets $\Om\subset\R^d,\,d\ge 3$ which are precompact. In that case the Newtonian capacity scales as a power $\beta_2=d-2$ of the homothety.

Throughout this paper we let $\Om$ be a non-empty, open, bounded set in Euclidean space $\R^d, \, d\ge 3$. For a set $A\subset\R^d$ we denote by $\overline{A}$ its closure, $\diam(A)=\sup\big\{|x-y|\ :\ x\in A,\ y\in A\big\}$ its diameter, and
$r(A)=\sup\big\{r\ge0\ :\ (\exists x\in A),\ (B_r(x)\subset A)\big\}$ its inradius, where $B_r(x)=\{y\in\R^d\ :\ |x-y|<r\}$ is the ball of radius $r$ centred at $x$. Before we state the main results we recall some basic facts about the torsion function, torsional rigidity, and Newtonian capacity.

The torsion function for an open set $\Om$ with finite measure is the solution of
$$-\Delta u=1,\quad u\in H_0^1(\Om),$$
and is denoted by $u_\Om$. It is convenient to extend $u_\Om$ to all of $\R^d$ by defining $u_\Om=0$ on $\R^d\setminus\Om$. It is well known that $u_\Om$ is non-negative, bounded (\cite{vdB,vdBC,GS,HV}), and monotone increasing with respect to $\Om$, that is
$$\Om_1\subset\Om_2\Rightarrow u_{\Om_1}\le u_{\Om_2}.$$
The torsional rigidity of $\Om$, or torsion for short, is denoted by
$$T(\Om)=\|u_{\Om}\|_1,$$
where $\|\cdot\|_p,\,1\le p\le \infty$ denotes the usual $L^p$ norm. It follows that
\begin{equation}\label{e4}
\Om_1\subset\Om_2\Rightarrow T(\Om_1)\le T(\Om_2),
\end{equation}
and that the torsion satisfies the scaling property
\begin{equation}\label{e5}
T(t\Om)=t^{d+2}T(\Om),\quad t>0.
\end{equation}
Moreover $T$ is additive on unions of disjoint families of open sets:
$$T(\cup_{i\in I}\Om_i)=\sum_{i\in I}T(\Om_i).$$
It is straightforward to verify that if $E(a)$, with $a=(a_1,a_2,\dots,a_d)\in\R_+^d$, is the ellipsoid
$$E(a)=\bigg\{x\in\R^d\ :\ \sum_{i=1}^d\frac{x_i^2}{a_i^2}<1\bigg\},$$
then
$$u_{E(a)}(x)=\frac12\bigg(\sum_{i=1}^d\frac{1}{a_i^2}\bigg)^{-1}\bigg(1-\sum_{i=1}^d\frac{x_i^2}{a_i^2}\bigg),$$
and
\begin{equation}\label{e9}
T(E(a))=\frac{\omega_d}{d+2}\bigg(\prod_{i=1}^d a_i\bigg) \bigg(\sum_{i=1}^d\frac{1}{a_i^2}\bigg)^{-1},
\end{equation}
where
$$\omega_d=\frac{\pi^{d/2}}{\Gamma((d+2)/2)}$$
is the Lebesgue measure of a ball $B_1$ with radius $1$ in $\R^d$. We put
$$\tau_d=T(B_1)=\frac{\omega_d}{d(d+2)}.$$
The de Saint-Venant inequality (see for instance Chapter V in \cite{PSZ}) asserts that
\begin{equation}\label{e12}
T(\Om)\le T(\Om^*),
\end{equation}
where $\Om^*$ is any ball with $|\Om|=|\Om^*|$. It follows by scaling that
\begin{equation}\label{e13}
\frac{T(\Om)}{|\Om|^{(d+2)/d}}\le\frac{\tau_d}{\omega_d^{(d+2)/d}}=\frac{1}{d(d+2)\omega_d^{2/d}}.
\end{equation}

Below we recall some basic facts about the Newtonian capacity $\cp(K)$ of a compact set $K\subset \R^d,\,d\ge 3$. There are several equivalent definitions of $\cp(K)$ of which we choose
$$\cp(K)=\inf\bigg\{\int_{\R^d}|D\varphi|^2\,dx\ :\ \varphi_{K_\vps}\ge1,\ \varphi\in C_0^1(\R^d),\vps>0\bigg\},$$
where $\varphi_{K_\vps}$ is the restriction of $\varphi$ to $K_{\vps}=\{x\in \R^d:\textup{dist}(x,K)<\vps\}$.
It follows that
\begin{equation}\label{e15}
K_1\subset K_2\Rightarrow \cp(K_1)\le\cp(K_2),
\end{equation}
and that the capacity satisfies the scaling property
\begin{equation}\label{e16}
\cp(tK)=t^{d-2}\cp(K),\,t>0.
\end{equation}
Moreover if $\{K_i,\,i\in I\}$ is a countable family of compact sets such that $\cup_{i\in I}K_i$ is compact, then
$$\cp(\cup_{i\in I}K_i)\le\sum_{i\in I}\cp(K_i).$$
It was reported in \cite{IMcK} pp. 260 that the Newtonian capacity of an ellipsoid was computed in volume 8, pp. 103-104 in \cite{GC}. The formula is given in terms of an elliptic integral, and reads
\begin{equation}\label{e18}
\cp\big(\overline{E(a)}\big)=\frac{4\pi^{d/2}}{\Gamma(d/2)}\mathfrak{e}(a)^{-1},
\end{equation}
where
\begin{equation}\label{e19}
\mathfrak{e}(a)=\int_0^{\infty}\bigg(\prod_{i=1}^d\big(a_i^2+t\big)\bigg)^{-1/2}dt.
\end{equation}
We put
$$\kappa_d=\cp(\overline{B_1})=\frac{4\pi^{d/2}}{\Gamma((d-2)/2)},$$
so that
\begin{equation}\label{e20}
\cp\big(\overline{E(a)}\big)=\frac{2\kappa_d}{d-2}\mathfrak{e}(a)^{-1}.
\end{equation}
The isoperimetric inequality for Newtonian capacity (see \cite{PSZ}) asserts that for all compact sets $K\subset \R^d$, $d\ge 3$,
$$\cp(K)\ge\cp(K^*),$$
where $K^*$ is a closed ball with $|K|=|K^*|$. It follows by scaling that
\begin{equation}\label{e22}
\frac{\cp(K)}{|K|^{(d-2)/d}}\ge \frac{\kappa_d}{\omega_d^{(d-2)/d}}.
\end{equation}
The shape functional we consider in the present paper is
\begin{equation}\label{e22a}G_q(\Om)=\frac{\cp(\overline{\Om})T(\Om)^q}{|\Om|^{1+q+2(q-1)/d}},
\end{equation}
defined for a bounded open set $\Om\subset\R^d$, $d\ge3$. By \eqref{e5} and \eqref{e16} we obtain that $G_q$ is scaling invariant. With the definitions above we have
$$G_q(B_1)= \frac{\kappa_d\tau_d^q}{\omega_d^{1+q+2(q-1)/d}}.$$

Since the ball $\Om^*$ with measure $|\Om^*|=|\Om|$ maximises the torsional rigidity $T(\Om)$ (de Saint-Venant), and its closure minimises the Newtonian capacity $\cp(\overline{\Om}),$ competition enters in the minimisation or maximisation problems for the functional in \eqref{e22a}.

\medskip

All of our main results are for $d\ge 3$, and are as follows.
\begin{theorem}\label{the1}
\begin{enumerate}
\item[\textup{(i)}]If $q\in \R$, then
$$\sup\{G_q(\Om)\ :\ \Om\textup{ open and bounded}\}=+\infty.$$
\item[\textup{(ii)}]If $q\le 0,$ then
\begin{equation}\label{e28}
\min\{G_q(\Om)\ :\ \Om\textup{ open and bounded}\}=G_q(B_1),
\end{equation}
with equality if and only if $\Om$ is (up to sets of capacity $0$) a ball in $\R^d$.
\item[\textup{(iii)}]If $q>0$, then
$$\inf\{G_q(\Om)\ :\ \Om\textup{ open and bounded}\}=0.$$
\end{enumerate}
\end{theorem}

\medskip

\begin{theorem}\label{the2}
\begin{enumerate}
\item[\textup{(i)}]If $q<1,$ then
$$\sup\{G_q(\Om)\ :\ \Om\textup{ open, bounded and convex}\}=+\infty.$$
\item[\textup{(ii)}]If $q\ge 1,$ then
\begin{align}\label{e25}
\sup\{G_q(\Om)\ :\ &\Om\textup{ open, bounded and convex}\}\nonumber\\
&\le\frac{2^{(d+2)/2}d^{3q-2+d(q+1)}}{d-2}G_q(B_1).
\end{align}
If $q>1$, then the variational problem in the left-hand side of \eqref{e25} has a maximiser, say $\Om^+$.
For any such maximiser,
\begin{equation}\label{e26}
\frac{\diam(\Om^+)}{r(\Om^+)}\le2d\bigg(\frac{2^{(d+2)/2}d^{(1+q)d+3q-2}}{d-2}\bigg)^{d/(2(q-1))}.
\end{equation}
\item[\textup{(iii)}]If $q=1$ and $d=3$, then the variational problem in the left-hand side of \eqref{e25} has a maximiser, say $\Om^+$. For any such maximiser,
\begin{equation}\label{e26a}\frac{\diam(\Om^+)}{r(\Om^+)}\le 2\cdot3^8e^{3^7}.
\end{equation}
\end{enumerate}
\end{theorem}

\medskip

\begin{theorem}\label{the3}
\begin{enumerate}
\item[\textup{(i)}]If $q>(d-2)/(2(d-1))$, then
\begin{equation}\label{e29}
\inf\{G_q(\Om)\ :\ \Om\textup{ open, bounded and convex}\}=0.
\end{equation}
\item[\textup{(ii)}]If $0<q\le (d-2)/(2(d-1))$, then
\begin{equation}\label{e30}
\inf\{G_q(\Om)\ :\ \Om\textup{ open, bounded and convex}\}\ge\frac{1}{2d^{d+(d+2)q}}G_q(B_1).
\end{equation}
If $0<q<(d-2)/(2(d-1))$, then the variational problem in the left-hand side of \eqref{e30} has a convex minimiser, say $\Om^-$. For any such minimiser,
\begin{equation}\label{e31}
\frac{\diam(\Om^-)}{r(\Om^-)}\le 2d\big(2d^{d+(d+2)q}\big)^{\frac{d(d-1)}{d-2-2q(d-1)}}.
\end{equation}
\end{enumerate}
\end{theorem}

We were unable to prove the existence or non-existence of a maximiser for the left-hand side of \eqref{e25} for $q=1$ and $d>3$. In these higher-dimensional cases there is a lack of compactness. For example if
$$a_{\vps}=(\underbrace{1,...,1,}_\text{$d-k$}\underbrace{\vps,...,\vps}_\text{$k$}),$$ and if $k\ge 3,$ then $\lim_{\vps\to0}G_1(E(a_{\vps}))$ exists and is strictly positive.  Similarly we were unable to prove the existence of a minimiser of the left-hand side of \eqref{e30} at the critical point $q=(d-2)/(2d-2)$.

\medskip
The proofs of Theorems \ref{the1}, \ref{the2}, and \ref{the3} are deferred to Sections \ref{sec4}, \ref{sec2}, and \ref{sec3} respectively. A key ingredient in these proofs is John's ellipsoid theorem \cite{J}. This theorem asserts that for any open, bounded convex set $\Om$ in $\R^d$ there exists a translation and rotation of $\Om$, again denoted by $\Om$, and an open ellipsoid $E(a)$ such that
\begin{equation}\label{e33}
E(a/d)\subset\Om\subset E(a).
\end{equation}
Moreover, among all ellipsoids in $\Om$, $E(a/d)$ has maximal measure.

Finally in Section \ref{sec5} we discuss the optimisation of a functional over all open bounded planar convex sets with fixed measure, and which involves the logarithmic capacity and torsional rigidity.

\section{Proof of Theorem \ref{the1}\label{sec4}}

\begin{proof}
To prove the assertion under (i) we let $\Om$ be the disjoint union of an open ball $B'$ of measure $1/2$ and an open ellipsoid $E(b_\vps)$, with $b_\vps=(L_\vps,\dots,L_\vps,\vps)$, of measure $1/2$, where
$$L_\vps=(2\omega_d\vps)^{1/(1-d)}.$$
We have, by \eqref{e18} and \eqref{e19},
\begin{align}
\cp\big(\overline{E(b_\vps})\big)&=\frac{4\pi^{d/2}}{\Gamma(d/2)}\Big(\int_0^\infty dt\,(L^2_\vps+t)^{(1-d)/2}(\vps^2+t)^{-1/2}\Big)^{-1}\nonumber\\ &\ge
\frac{4\pi^{d/2}}{\Gamma(d/2)}\Big(\int_0^\infty\, dt\,(L^2_\vps+t)^{(1-d)/2}t^{-1/2}\Big)^{-1}\nonumber\\ &
=\frac{4\pi^{(d-1)/2}\Gamma((d-1)/2)}{\Gamma(d/2)\Gamma((d-2)/2)}L_{\vps}^{d-2},\nonumber
\end{align}
where we have used formulae 8.380.3 and 8.384.1 in \cite{GR}.
Hence
\[\begin{split}
G_q(\Om)&=\cp\big(\overline{B'}\cup \overline{E(b_\vps)}\big)T\big(B'\cup E(b_\vps)\big)^q\\
&\ge\cp\big(\overline{E(b_\vps)}\big)T(B')^q\\
&\ge\frac{4\pi^{(d-1)/2}\Gamma((d-1)/2)}{\Gamma(d/2)\Gamma((d-2)/2)}T(B')^q(2\omega_d\vps)^{(d-2)/(1-d)},
\end{split}\]
which tends to $+\infty$ as $\vps\downarrow0.$

To prove the assertion under (ii), we recall \eqref{e12}, and infer that $T^q(\Om)\ge T^q(\Om^*)$ for $q\le0$. This implies \eqref{e28} by \eqref{e13} and \eqref{e22}.

To prove (iii), we let $Q\subset\R^d$ be a cube with $|Q|=1$. Let $N\in\N$ be arbitrary. The cube $Q$ contains $N^d$ open disjoint cubes each of measure $N^{-d}$. Each open cube contains an open ball with radius $1/(2N)$. Let $Q_N$ be the union of these $N^d$ open balls. Since $Q_N\subset Q$ we have $\cp(\overline{Q_N})\le\cp(\overline{Q})$. On the other hand, additivity and scaling properties of the torsion give
$$T(Q_N)=N^d(2N)^{-(d+2)}T(B_1)=2^{-d-2}N^{-2}T(B_1).$$
Furthermore,
$$|Q_N|=\frac{\omega_d}{2^d}. $$
Hence
\begin{align*}
\inf\big\{G_q(\Om)\ :\ \Om\textup{ open and bounded}\big\}
&\le\frac{\cp(\overline{Q})T^q(Q_N)}{|Q_N|^{1+q+2(q-1)/d}}\nonumber\\
&=\frac{2^{d-2}}{\omega_d^{1+q+2(q-1)/d}}\cp(\overline{Q})T^q(B_1)N^{-2q}.
\end{align*}
This implies \eqref{e29} since $q>0$, and $N\in\N$ was arbitrary.
\end{proof}

\section{Proof of Theorem \ref{the2} \label{sec2}}

\begin{proof}
To prove (i) we consider the open ellipsoid $E(a_\vps)$ with $a_\vps=(1,\vps,\dots,\vps)$. We have
$$|E(a_\vps)|=\omega_d\vps^{d-1},$$
$$T\big(E(a_\vps)\big)=\frac{\omega_d}{d+2}\frac{\vps^{d+1}}{d-1+\vps^2},$$
\begin{equation*}
\cp\big(\overline{E(a_\vps})\big)=\begin{cases}4\pi\varepsilon\big(\log(\varepsilon^{-1})\big)^{-1}(1+o(1)),&d=3,\ \varepsilon\downarrow0,\\
\ds\frac{2\pi^{d/2}(d-3)}{\Gamma(d/2)}\varepsilon^{d-3}(1+o(1)),&d>3,\ \varepsilon\downarrow0,
\end{cases}
\end{equation*}
where we have used the formulae on p.260 in \cite{IMcK}. Hence
$$
G_q\big(E(a_\vps)\big)=\begin{cases}
C_3\vps^{2(q-1)/3}\big(\log\vps^{-1}\big)^{-1}(1+o(1)),&d=3,\ \varepsilon\downarrow0,\\
C_d\vps^{2(q-1)/d}(1+o(1)),&d>3,\ \varepsilon\downarrow0.
\end{cases}
$$
where $C_d$ is a positive constant depending only on $d$. Since $q<1$ we obtain the desired result by letting $\vps\downarrow0$.

To prove (ii) we first observe that the formulae for $|E(a)|, \cp(\overline{E(a)}),$ and $T(E(a))$ are symmetric in the $a_i$'s. Without loss of generality we may therefore assume here, and throughout this paper, that $a_1\ge a_2\ge...\ge a_d$. By inclusion,
and \eqref{e33} we have
\begin{equation}\label{e34}
d^{-d}\omega_d\prod_{i=1}^d a_i=|E(a/d)|\le|\Om|\le|E(a)|=\omega_d\prod_{i=1}^d a_i,
\end{equation}
and
\begin{equation}\label{e35}
T(\Om)\ge T(E(a/d))=\frac{\omega_d}{d^{d+2}(d+2)}\bigg(\prod_{i=1}^d a_i\bigg)\bigg(\sum_{i=1}^d\frac{1}{a_i^2}\bigg)^{-1}\ge\frac{\tau_d}{d^{d+2}}\bigg(\prod_{i=1}^d a_i\bigg)a_d^2.
\end{equation}
We have by \eqref{e18},
\begin{align}\label{e42}
\mathfrak{e}(a)
&\ge\int_0^{a_d^2}dt\,\bigg(\prod_{i=1}^d\big(a_i^2+t\big)\bigg)^{-1/2}\nonumber\\
&\ge a_d^2\bigg(\prod_{i=1}^d\big(a_i^2+a_d^2\big)\bigg)^{-1/2}\nonumber\\
&\ge2^{-d/2}a_d^2\bigg(\prod_{i=1}^d a_i^2\bigg)^{-1/2}\nonumber \\ &
=2^{-d/2}\bigg(\prod_{i=1}^d a_i\bigg)^{-1}a_d^2.
\end{align}
By \eqref{e19}, \eqref{e20} and \eqref{e42}, taking into account that $\Gamma(z+1)=z\Gamma(z)$, $z>0$,
\begin{equation}\label{e43}
\cp(\overline{\Om})\le \cp(\overline{E(a)})\le \frac{2^{(d+2)/2}\kappa_d}{d-2}\bigg(\prod_{i=1}^d a_i\bigg)a_d^{-2}.
\end{equation}
By \eqref{e4} and \eqref{e9},
\begin{equation}\label{e44}
T(\Om)\le T(E(a))\le d\tau_d\bigg(\prod_{i=1}^d a_i\bigg)a_d^2.
\end{equation}
By \eqref{e34}, \eqref{e43}, \eqref{e44}, $a_1\ge a_2\ge...\ge a_d$, and $q>1$ we obtain,
\begin{align}\label{e45}
G_q(\Om)
&\le\frac{\cp(\overline{E(a)})T(E(a))^q}{|E(a/d)|^{1+q+2(q-1)/d}}\nonumber\\
&\le\frac{2^{(d+2)/2}\kappa_d}{d-2}\big(d\tau_d\big)^q\big(d^{-d}\omega_d\big)^{-(1+q+2(q-1)/d)}\bigg(\prod_{i=1}^d a_i\bigg)^{2(1-q)/d}a_d^{2q-2}\nonumber\\
&=\frac{2^{(d+2)/2}d^{3q-2+d(q+1)}}{d-2}G_q(B_1)\bigg(\prod_{i=1}^d a_i\bigg)^{2(1-q)/d}a_d^{2q-2}\nonumber\\
&\le\frac{2^{(d+2)/2}d^{3q-2+d(q+1)}}{d-2}G_q(B_1).
\end{align}
This proves \eqref{e25}.

To prove the existence of a maximiser, we observe that if the left-hand side of \eqref{e25} equals $G_q(B_1)$ then $B_1$ is a maximiser which satisfies \eqref{e26}.
If the left-hand side of \eqref{e25} is strictly greater than $G_q(B_1)$, we let $\Om$ be bounded, open, and convex, and such that
\begin{equation}\label{e46a}
G_q(\Om)> G_q(B_1).
\end{equation}
By the third inequality in \eqref{e45}, \eqref{e46a}, $q>1$, and $a_1\ge a_2\ge...\ge a_d$, we find that
\begin{equation}\label{e47}
a_1\le \beta_d^{d/(2(q-1))} a_d,
\end{equation}
where $\beta_d$ is the coefficient of $G_q(B_1)$ in the right-hand side of \eqref{e45}. Since
\begin{equation}\label{e48}
\diam(\Om)\le\diam(E(a))\le 2a_1,
\end{equation}
and
\begin{equation}\label{e49}
r(\Om)\ge r(E(a/d))=\frac{a_d}{d},
\end{equation}
we obtain by \eqref{e47}--\eqref{e49},
\begin{equation}\label{e50}
\frac{\diam(\Om)}{r(\Om)}\le 2d\bigg(\frac{2^{(d+2)/2}d^{3q-2+d(q+1)}}{d-2}\bigg)^{d/(2(q-1))}.
\end{equation}

Let $(\Om_n)$ be a maximising sequence for the left-hand side of \eqref{e25}. Since this supremum is scaling invariant we fix $r(\Om_n)=1$. By \eqref{e50}, $\diam(\Om_n)\le L$ for some $L<\infty$, and for all $n$.
By taking translations of $(\Om_n)$ these translates are contained in a closed ball $B_L$ of radius $L$. Since the Hausdorff metric is compact on the space of convex, compact sets in $B_L,$ there exists a subsequence of $(\overline{\Om_n})$, again denoted by $(\overline{\Om_n})$ which converges in the Hausdorff (and in the complementary Hausdorff) metric to an element say $\tilde{\Om}^+$. Set $\Om^+=\textup{int}(\tilde{\Om}^+).$ Note that $\Om^+$ is an open, bounded, convex set which is non-empty since $\Om^+$ has inradius $1$. Furthermore measure, torsion, capacity, and diameter are all continuous with respect to this metric. Hence
$$G_q(\Om^+)=\lim_{n\to\infty}G_q(\Om_n),$$
and $\Om^+$ is a maximiser which satisfies \eqref{e26}.

To prove (iii) we let $q=1$ and $d=3$.
Let $\Om$ be an element of a maximising sequence. We may assume that
\begin{equation}\label{e36f}
G_1(B_1)\le G_1(\Om)\le\frac{\cp\big(\overline{E(a)}\big)T\big(E(a)\big)}{\big|E(a/3)\big|^2}.
\end{equation}
We obtain an upper bound on $\cp(\overline{E(a)})$ by obtaining a lower bound on $\mathfrak{e}(a)$. By \eqref{e19}, we have
\begin{align}\label{e36a}
\mathfrak{e}(a_1,a_2,a_3)&\ge\int_0^{\infty}\big(a_1^2+t\big)^{-1/2}\big(a_2^2+t\big)^{-1}dt\nonumber\\
&=\frac{2}{\big(a_1^2-a_2^2\big)^{1/2}}\log\bigg(\frac{a_1}{a_2}+\Big(\frac{a_1^2}{a_2^2}-1\Big)^{1/2}\bigg)\nonumber\\
&\ge\frac{2}{a_1}\log\big(a_1a_2^{-1}\big).
\end{align}
By \eqref{e18} for $d=3$, and \eqref{e36a},
$$\cp(\overline{E(a)})\le \kappa_3 a_1\big(\log\big(a_1a_2^{-1}\big)\big)^{-1}.$$
Since $\Om\subset B_{a_1}$, we also have
$$\cp(\overline{\Om})\le\cp(\overline{B_{a_1}})=\kappa_3 a_1.$$
Hence
$$\cp(\overline{\Om})\le \kappa_3a_1\min\big\{1,\big(\log\big(a_1a_2^{-1}\big)\big)^{-1}\big\}.$$
In addition,
$$T(E(a))\le3\tau_3a_1a_2a_3^3\;,\qquad|E(a/3)|=3^{-3}\omega_3a_1a_2a_3\;.$$
Summarising, from \eqref{e36f} we obtain
\begin{align}\label{e36g}
G_1(\Om)\le3^7G_1(B_1)\cdot\frac{a_3}{a_2}\cdot\min\{1, \big(\log\big(a_1a_2^{-1}\big)\big)^{-1}\}.
\end{align}
If the supremum in the left-hand side of \eqref{e25} equals $G_1(B_1)$, then $B_1$ is a maximiser which satisfies \eqref{e26a}. If not then we may assume that $G_1(\Om)>G_1(B_1)$. This, together with \eqref{e36g}, yields
$a_3\ge 3^{-7}a_2\;, a_1\le a_2e^{3^7}.$ These inequalities imply that
$a_1/a_3\le 3^7e^{3^7}.$ Hence \eqref{e48} and \eqref{e49} yield  $$\frac{\diam(\Om)}{r(\Om)}\le 2\cdot 3^8e^{3^7}.$$
The remaining part of the proof follows similar lines as those in the proof of (ii).
\end{proof}

\section{Proof of Theorem \ref{the3} \label{sec3}}

\begin{proof}
To prove (i), we consider as $\Om$ the ellipsoid $E(a)$ with
\begin{equation*}
a=(a_1,\dots,a_1,a_d),\qquad a_1\ge a_d,\qquad a_1^{d-1}a_d=1,
\end{equation*}
where $a_d\in(0,1)$ is arbitrary. Since $E(a)\subset B_{a_1}$ we have by \eqref{e15}, \eqref{e16} and \eqref{e20},
\begin{equation}\label{e55}
\cp(\overline{E(a)})\le \kappa_da_1^{d-2}.
\end{equation}
By \eqref{e9} and \eqref{e55},
\begin{align*}
G_q(E(a))\le\frac{\kappa_d a_1^{d-2}}{\omega_d^{1+q+2(q-1)/d}}\Big(\frac{\omega_d}{d+2}\frac{a_1^2a_d^2}{(d-1)a_d^2+a_1^2}\Big)^q\le\frac{\kappa_d a_1^{d-2}a_d^{2q}}{\omega_d^{1+2(q-1)/d}(d+2)^q}.
\end{align*}
Since $a_1=a_d^{-1/(d-1)}$ we have
$$G_q(E(a))\le\frac{\kappa_d}{\omega_d^{1+2(q-1)/d}(d+2)^q}a_d^{-(d-2)/(d-1)+2q},$$
and since $a_d\in(0,1)$ was arbitrary we obtain \eqref{e29}.

To prove (ii) we let $a_1\ge a_2\ge...\ge a_d$. We have
\begin{align}\label{e60}
\cp(\overline{\Om})\ge\cp(\overline{E(a/d)})=d^{2-d}\cp(\overline{E(a)})
=\frac{2\kappa_d}{d^{d-2}(d-2)}\big(\mathfrak{e}(a)\big)^{-1}.
\end{align}
In order to obtain an upper bound on $\mathfrak{e}(a)$ we have by using the inequality $(x^2+t)^{1/2}\ge 2^{-1/2}(x+t^{1/2})$, and the change of variables $t=\theta^2$,
\begin{align}\label{e61}
\mathfrak{e}(a)&\le \bigg(\prod_{i\le d-3}a_i^{-1}\bigg)\int_0^{\infty}dt\,\big((a_{d-2}^2+t)(a_{d-1}^2+t)(a_{d}^2+t)\big)^{-1/2}\nonumber \\ &
\le\bigg(\prod_{i\le d-3}a_i^{-1}\bigg)\int_0^{\infty}dt\,\big((a_{d-2}^2+t)(a_{d-1}^2+t)t\big)^{-1/2}\nonumber \\ &
\le 2\bigg(\prod_{i\le d-3}a_i^{-1}\bigg)\int_0^{\infty}dt\,\big((a_{d-2}+t^{1/2})(a_{d-1}+t^{1/2})t^{1/2}\big)^{-1}\nonumber \\ &
=4\bigg(\prod_{i\le d-3}a_i^{-1}\bigg)\int_0^{\infty}d\theta\,\big((a_{d-2}+\theta)(a_{d-1}+\theta)\big)^{-1}\nonumber \\ &
=4\bigg(\prod_{i\le d-2}a_i^{-1}\bigg)\Big(1-\frac{a_{d-1}}{a_{d-2}}\Big)^{-1}\log(a_{d-2}/a_{d-1}),
\end{align}
where the product over the empty set in the right-hand side of \eqref{e61} is defined to be equal to $1$, and where the case $a_{d-2}=a_{d-1}$ follows by taking the appropriate limit in the right-hand side of \eqref{e61}.
It is elementary to verify that
$$(1-x)^{-1}\log (x^{-1})\le \log(e/x), \quad 0<x< 1, $$
and
$$\lim_{x\uparrow 1}(1-x)^{-1}\log (x^{-1})=1.$$
This gives by \eqref{e61},
\begin{align}\label{e64}
\mathfrak{e}(a)\le4\bigg(\prod_{i\le d-2}a_i^{-1}\bigg)\log(ea_{d-2}/a_{d-1}).
\end{align}
Hence by \eqref{e60} and \eqref{e64},
\begin{equation}\label{e65}
\cp(\overline{\Om})
\ge\frac{\kappa_d}{2d^{d-2}(d-2)}\bigg(\prod_{i\le d-2}a_i\bigg)\big(\log(ea_{d-2}/a_{d-1})\big)^{-1}.
\end{equation}
By \eqref{e34}, \eqref{e35}, and \eqref{e65},
\begin{equation}\label{e66}
G_q(\Om)\ge \frac{1}{2d^{d-2+(d+2)q}(d-2)}G_q(B_1)\frac{a_d^{2q-1}}{a_{d-1}}\bigg(\prod_{i=1}^da_i\bigg)^{2(1-q)/d}\big(\log(ea_{d-2}/a_{d-1})\big)^{-1}.
\end{equation}
The $a$-dependence in the right-hand side of \eqref{e66} is scaling invariant. It is convenient to choose $\prod_{i=1}^da_i=1$. We then have
$$\frac{a_{d-2}}{a_{d-1}}=\bigg(\prod_{i\le d-3}a_i^{-1}\bigg)a_{d-1}^{-2}a_d^{-1}\le a_{d-1}^{-(d-1)}a_d^{-1}.$$
This gives with \eqref{e66},
\begin{equation}\label{e67}
G_q(\Om)\ge \frac{1}{2d^{d-2+(d+2)q}(d-2)}G_q(B_1)\frac{a_d^{2q-1}}{a_{d-1}}\big(\log\big(e/(a_{d-1}^{d-1}a_d)\big)\big)^{-1}.
\end{equation}
Since
$$x^{1/(d-1)}\log(e/x)\le d-1,\quad 0<x<1,$$
we have, with $x=a_{d-1}^{d-1}a_d$,
$$\big(\log\big(e/(a_{d-1}^{d-1}a_d)\big)\big)^{-1}\ge a_{d-1}a_d^{1/(d-1)}(d-1)^{-1}.$$
This, together with \eqref{e67}, gives
\begin{align}\label{e68}
G_q(\Om)&\ge \frac{1}{2d^{d-2+(d+2)q}(d-2)(d-1)}G_q(B_1)a_d^{2q-\frac{d-2}{d-1}}\nonumber \\ &
\ge \frac{1}{2d^{d+(d+2)q}}G_q(B_1)a_d^{2q-\frac{d-2}{d-1}}.
\end{align}
This proves \eqref{e30} since $a_d\in(0,1]$, and $q\le(d-2)/(2(d-1))$.

To prove the existence of a minimiser, we observe that if the left-hand side of \eqref{e30} equals $G_q(B_1)$ then $B_1$ is a minimiser which satisfies \eqref{e31}.
If the left-hand side of \eqref{e30} is strictly less than $G_q(B_1)$, we let $\Om$ be bounded and convex such that
\begin{equation}\label{e46}
G_q(\Om)< G_q(B_1).
\end{equation}
By \eqref{e68} and \eqref{e46} we infer
\begin{equation}\label{e69}
a_d\ge \bigg(\frac{1}{2d^{d+(d+2)q}}\bigg)^{(d-1)/(d-2-2q(d-1))}.
\end{equation}
Since $\prod_{i=1}^da_i=1$, and $a_1\ge a_2\ge...\ge a_d$ we have $a_1\le a_d^{1-d}$. By \eqref{e48}, \eqref{e49} and \eqref{e69} we obtain,
\begin{equation}\label{e70}
\frac{\diam(\Om)}{r(\Om)}\le 2da_d^{-d}\le 2d\big(2d^{d+(d+2)q}\big)^{\frac{d(d-1)}{d-2-2q(d-1)}}.
\end{equation}
The proof of the existence of a minimiser is similar to the proof of the existence of a maximiser in Theorem \ref{the1}(ii), and has been omitted. If $\Om^-$ is a minimiser then, by continuity of diameter and inradius,
$\Om^-$ satisfies \eqref{e70}. This proves \eqref{e31}.
\end{proof}

\section{The logarithmic capacity \label{sec5}}

We briefly recall some basic properties of the logarithmic capacity of a compact set $K$ in $\R^2$.
Let $\mu$ be a probability measure supported on $K$, and let

\begin{equation*}
I(\mu)=\iint_{K\times K}\log\Big(\frac{1}{|x-y|}\Big)\mu(dx)\mu(dy).
\end{equation*}
Furthermore let
\begin{equation*}
V(K)=\inf\big\{I(\mu):\mu \textup{ a probability measure on $K$} \big\}.
\end{equation*}
The logarithmic capacity of $K$ is denoted by $\cp(K)$, and is the non-negative real number
\begin{equation*}
\cp(K)=e^{-V(K)}.
\end{equation*}
It shares some of the properties of the Newtonian capacity. In particular if $K_1$ and $K_2$ are compact sets in $\R^2$ with  $K_1\subset K_2$ then $\cp(K_1)\le \cp(K_2)$. Moreover, $\cp(K)$ is invariant under translations and rotations of $K$, and
\begin{equation}\label{f4}
\cp(K)\ge \cp(K^*),
\end{equation}
where $K^*$ is the disc with $|K|=|K^*|$. See \cite{AHN} for some refinements.
Finally for a homothety,
\begin{equation}\label{f4}
\cp(tK)=t\,\cp(K),\quad t>0.
\end{equation}
The classic treatise \cite{L} gives various planar domains for which the logarithmic capacity can be computed analytically. In particular for the ellipse with semi axes $a_1$ and $a_2$,
\begin{equation}\label{f5}
\cp(\overline{E(a_1,a_2)})=\frac12(a_1+a_2).
\end{equation}
For an open, bounded, convex planar set $\Om$ we define the functional
\begin{equation*}
H_q(\Om)=\frac{\cp(\overline{\Om})T^q(\Om)}{|\Om|^{(1+4q)/2}}.
\end{equation*}
In particular we have
\begin{equation*}
H_q(B_1)=\frac{\tau_2^q}{\omega_2^{(1+4q)/2}}=\frac{1}{8^q\pi^{q+1/2}}.
\end{equation*}
We immediately see that by \eqref{e5}, and \eqref{f4} that $H_q(t\Om)=H_q(\Om),\, t>0.$ Our main result is the following.
\begin{theorem}\label{the4}
\begin{enumerate}
\item[\textup{(i)}]If $q\ge1/2$, then
\begin{equation}\label{f7}
\sup\{H_q(\Om)\ :\ \Om\textup{ open, bounded, planar, and convex}\}
\le2^{1+5q}H_q(B_1).
\end{equation}
\item[\textup{(ii)}]If $q>1/2$ then the left-hand side of \eqref{f7} has an open, bounded, planar, and convex maximiser. For any such maximiser, say $\Om^+$,
\begin{equation}\label{f8}
\frac{\diam(\Om^+)}{r(\Om^+)}\le\frac{2^{14q}}{2q-1}.
\end{equation}
\item[\textup{(iii)}]If $q<1/2$, then
\begin{equation}\label{f9}
\sup\{H_q(\Om)\ :\ \Om\textup{ open, bounded, planar, and convex}\}=+\infty.
\end{equation}
\item[\textup{(iv)}]If $q\le1/2$, then
\begin{equation}\label{f10}
\inf\{H_q(\Om)\ :\ \Om\textup{ open, bounded, planar, and convex}\}\ge 2^{-2(1+2q)}H_q(B_1).
\end{equation}
\item[\textup{(v)}]If $q<1/2$ then the left-hand side of \eqref{f10} has an open, bounded, planar, and convex minimiser. For any such minimiser, say $\Om^-$,
\begin{equation}\label{f11}
\frac{\diam(\Om^-)}{r(\Om^-)}\le 2^{2(3+2q)/(1-2q)}.
\end{equation}
\item[\textup{(vi)}]If $q>1/2$, then
\begin{equation*}
\inf\{H_q(\Om)\ :\ \Om\textup{ open, bounded, planar, and convex}\}=0.
\end{equation*}
\end{enumerate}
\end{theorem}
\begin{proof}
(i) If $E(a)$ is the John's ellipsoid for $\Om$ then,  $E(a/2)\subset \Om\subset E(a)$ with $a_1\ge a_2$. Furthermore,
$$\cp(\overline{E(a)})\le a_1,\qquad T(E(a))\le 2\tau_2a_1a_2^3,\qquad|\Om|\ge |E(a/2)|= \omega_2a_1a_2/4,$$
so that
\begin{equation}\label{f13}
H_q(\Om)\le 2^{1+5q}\frac{\tau_2^q}{\omega_2^{(1+4q)/2}}\Big(\frac{a_2}{a_1}\Big)^{q-1/2}.
\end{equation}
This implies \eqref{f7} since $q\ge1/2$.

\medskip

\noindent(ii) To prove \eqref{f8} we have that
either the supremum in the left-hand side of \eqref{f7} is attained for a ball, in which case the maximiser exists and satisfies \eqref{f8}, or we may assume that $H_q(\Om)> H_q(B_1)$.
This implies, by \eqref{e48} and \eqref{e49}, that
\begin{equation*}
\frac{\diam(\Om)}{r(\Om)}\le \frac{2^{14q}}{2q-1}.
\end{equation*}
The remaining part of the proof is similar to the corresponding parts in the proof of Theorem \ref{the2}.

\medskip

\noindent(iii) By \eqref{e34}, \eqref{e35}, and \eqref{f5},
\begin{equation}\label{f16}
H_q(\Om)\ge 2^{-2(1+2q)}   \frac{\tau_2^q}{\omega_2^{(1+4q)/2}}  \Big(\frac{a_2}{a_1}\Big)^{q-1/2}.
\end{equation}
This implies \eqref{f9} by letting $a_2/a_1\rightarrow 0$ in \eqref{f16}.

\medskip

\noindent(iv) This follows from \eqref{f16} and $a_1\ge a_2$.

\medskip

\noindent(v) Either the infimum in the left-hand side of \eqref{f10} is attained for a ball, in which case the minimiser exists and satisfies \eqref{f11}, or we may assume that $H_q(\Om)< H_q(B_1)$. By \eqref{f16}, \eqref{e48}, and \eqref{e49}
\begin{equation*}
\frac{\diam(\Om)}{r(\Om)}\le 2^{\frac{2(3+2q)}{1-2q}}.
\end{equation*}
The remaining part of the proof is similar to the corresponding parts in the proof of Theorem \ref{the2}.

\medskip

\noindent(vi) This  follows by letting $a_2/a_1\rightarrow 0$ in \eqref{f13}.
\end{proof}
\section*{Acknowledgements} MvdB acknowledges support by the Leverhulme Trust through Emeritus Fellowship EM-2018-011-9. He is also grateful for hospitality at the Scuola Normale Superiore di Pisa. The work of GB is part of the project 2017TEXA3H {\it``Gradient flows, Optimal Transport and Metric Measure Structures''} funded by the Italian Ministry of Research and University.

\end{document}